\documentclass[12pt,a4paper,reqno]{amsart}
\usepackage{amscd}
\usepackage{amssymb}
\usepackage{latexsym}
\usepackage{amsmath}
\usepackage{mathrsfs}
\usepackage[X2,T1]{fontenc}
\usepackage{cite}

\theoremstyle{plain}
 \newtheorem{thm}{Theorem}[section]
 
 \newtheorem{lem}{Lemma}[section]
 
\theoremstyle{definition}
 
 \newtheorem{rem}{Remark}[section]
 
\numberwithin{equation}{section}

\renewcommand{\le}{\leqslant}\renewcommand{\leq}{\leqslant}
\renewcommand{\ge}{\geqslant}\renewcommand{\geq}{\geqslant}
\renewcommand{\setminus}{\smallsetminus}
\newcommand{\eqrefs}[3]{\eqref{#1}-\eqref{#3}}
\newcommand{\rr}[1]{\mathbf R^{#1}}
\newcommand{\R}{\mathbf R}
\newcommand{\N}{\mathbf N}
\newcommand{\zz}[1]{\mathbf Z^{#1}}

\newcommand{\loc}{\operatorname{loc}}
\newcommand{\eabs}[1]{\langle #1\rangle}
\newcommand{\nm}[2]{\|#1\|_{#2}}

\title{Generalized free time-dependent Schrödinger equation with initial data in Fourier Lebesgue spaces}



\author[Johansson]{\bfseries Karoline Johansson}

\address{
School of Computer science, Physics and Mathematics \\ 
Linnaeus University  \\ 
Vejdes Plats 6,7 \\
S-351 95 V\"axj\"o\\
Sweden}
\email{karoline.johansson@lnu.se}

\begin{document}

\vspace{18mm}
\setcounter{page}{1}
\thispagestyle{empty}

\begin{abstract}
Consider the solution of the free time-dependent Schr{\"o}\-dinger
equation with initial data $f$. It is shown by Sj\"ogren and
Sj\"olin \cite{artikel} that there exists $f$ in the Sobolev space
$H^s(\rr d), \; s=d/2$ such that tangential convergence can not be
widened to convergence regions. In \cite{johansson1} we obtain the corresponding
results for a generalized version of the Schrödinger
equation, where $-\Delta_x$ is replaced by an operator
$\varphi(D)$, with special conditions on $\varphi$. In this paper
we show that similar results may be obtained for initial data in
Fourier Lebesgue spaces.
\end{abstract}

 \maketitle

\section{Introduction}
\par
In this paper we establish non-existence results of non-tangential
convergence for the solution $u=S^{\varphi}f$ to the generalized time-dependent
Schr{\"o}dinger equation
\begin{equation}\label{generaliserad}
(\varphi(D)+i\partial_t )u =0,
\end{equation}
with the initial condition
\begin{equation*}
u(x,0)=f(x).
\end{equation*}
Here $\varphi$ is real-valued, and its radial derivatives of first
and second orders ($\varphi' =\varphi'_r$ and $\varphi''
=\varphi''_{rr}$) are continuous outside a compact set containing
origin, and fulfill appropriate growth conditions. In particular
$\varphi(\xi)=|\xi|^a$ will satisfy these conditions, for $a> 1$.
The exact conditions of admissible functions are given later on
and we refer to \cite{johansson1} for further examples of
admissible functions $\varphi$. By non-tangential convergence we
mean convergence to initial data as time goes to zero
and the space variable depends on the time non-linearly, i.e. the
space variable is not fixed (as for convergence along vertical
lines), nor linearly dependent of time (as for convergence along
arbitrary straight lines). Furthermore, we consider initial datas
in the weighted Fourier Lebesgue space, $\mathscr F
\!L^p_s=\mathscr F \!L^p_{(\omega)}$, where $\omega(x,\xi)=\eabs
\xi^s=(1+|\xi|^2)^{s/2}$, as well as for mixed weighted Fourier
Lebesgue spaces, 
$$
\mathscr F\!L_{s_1,s_2}^{p,q}(\rr{d_1}\times
\rr{d_2})=\mathscr F\!L_{s_1,s_2}^{p,q}(\rr d)=\mathscr
F\!L_{(\omega)}^{p,q}(\rr d),
$$
where $\omega(x,\xi)=\omega(x,\xi_1,\xi_2)=\eabs{\xi_1}^{s_1}\eabs{\xi_2}^{s_2}$ and $\xi_1\in\rr{d_1}$, $\xi_2\in\rr{d_2}$.

\par

 For $p=2$ and initial data which belongs to $\mathscr F\!L_{d(p-1)/p}^{p}(\rr d)$ we recover Theorem $1.1$ in \cite{johansson1}. There we proved existence of a function $f$
in the Sobolev space $H^{d/2}=\mathscr F \!L^2_{d/2}$ such that
near the vertical line $t \mapsto (x,t)$ through an arbitrary
point $(x,0)$ there are points accumulating at $(x,0)$ such that
the solution of equation \eqref{generaliserad} takes values far
from $f$. This means that the solution of the time-dependent
Schr{\"o}dinger equation with initial condition $u(x,0)=f(x)$ does
not converge non-tangentially to $f$. Therefore we can not
consider regions of convergence.
\par

In this paper, we prove that the corresponding results hold for functions $f\in \mathscr
F\!L^p_{d(p-1)/p}(\rr d)$ for $p\in (1,\infty]$. In the proof we use some ideas by
Sj{\"o}gren and Sj{\"o}lin in \cite{artikel} as well as
\cite{johansson1}, to construct
a counter example. Some ideas can also be found in Sj{\"o}lin
\cite{{LP}, {Counter}} and Walther \cite{{Sharp},{Sharpmax}}, and
some related results are given in Bourgain \cite{Bourgain}, Kenig,
Ponce and Vega \cite{Kenig}, and Sj{\"o}lin \cite{{ref2},{Hom}}.
The result in \cite{johansson1} is a special case of the result
obtained here and the techniques here are similar.

\par

Existence of regions of convergence has been studied before for
other equations. For example, Stein and Weiss consider in
\cite[Chapter II Theorem $3.16$]{SteinWeiss} Poisson integrals acting on
Lebesgue spaces. These
operators are related to the operator $S^{\varphi}$.

\par

%
%
%
For an appropriate function
$\varphi$ on $\rr d$, let $S^{\varphi}$ be the operator acting on
functions $f$ defined by
\begin{equation}\label{allmSaf}
f \mapsto \mathscr{F}^{-1}(\exp(i t \varphi(\xi))\mathscr{F}f),
\end{equation}
where $\mathscr{F}f$  is the Fourier transform of $f$, which takes the form
\begin{equation*}
\widehat{f}(\xi)=\mathscr{F}f(\xi)\equiv\int_{\rr d}e^{-i x \cdot \xi}f(x)\, d x,
\end{equation*}
when $f\in L^1(\rr d)$. This means that, if $\widehat{f}$ is an
integrable function, then $S^{\varphi}$ in \eqref{allmSaf} takes
the form
\begin{equation*}
S^{\varphi} f(x,t) = \frac{1}{(2 \pi)^d} \int_{\rr d} {e^{i x\cdot
\xi}e^{it \varphi(\xi)}\widehat{f}(\xi)}\, d \xi,\quad x \in
\rr d,\quad t \in \R.
\end{equation*}

\par

If $\varphi(\xi)=|\xi|^2$ and $f$ belongs to the Schwartz class
$\mathscr{S}(\rr d)$, then $S^{\varphi} f$ is the solution to the
time-dependent Schr{\"o}dinger equation $(-\Delta_x +i\partial_t)u
=0$ with the initial condition $u(x,0) = f(x)$.

\par

For more general
appropriate $\varphi$, for which the equation
\eqref{generaliserad} is well-defined, the expression $S^{\varphi}
f$ is the solution to the generalized time-dependent
Schr{\"o}dinger equation \eqref{generaliserad} with the initial
condition $u(x,0) = f(x)$. Note here that $S^{\varphi}f$ is
well-defined for any real-valued measurable $\varphi$ and $f\in \mathscr{S}$. On the other hand, it might be difficult to interpret \eqref{generaliserad} if for example $\varphi\not\in L^1_{loc}$.

\par

In order to state the main result we need to specify the
conditions on $\varphi$ and give some definitions. The function $\varphi$ should satisfy the
conditions
\begin{gather}\label{funktionerna}
\liminf_{r\to \infty}(\inf_{|\omega|=1}|\varphi'(r,\omega)|)=\infty,
\end{gather}
and
\begin{gather}\label{deriverade funktioner}
\sup_{r\ge
R}\Big(\sup_{|\omega|=1}\frac{r^{\beta}|\varphi''(r,\omega)|}{|\varphi'(r,\omega)|^2}\Big)<C,
\end{gather}
for some $\beta>0$ and some constant $C$.
Here $\varphi'(r\omega)=\varphi'(r,\omega)$ denotes the derivative
of $\varphi(r,\omega)$ with respect to $r$, and similarly for higher orders of derivatives.

\par
 In particular, $\varphi(\xi)= |\xi|^a$ is an appropriate
function for $a>1$ and $S^{\varphi}f(x,t)$ is then the solution to
the generalized time-dependent Schr{\"o}dinger equation
$((-\Delta_x)^{a/2} +i\partial_t)u =0$. For $a=2$ this is the
solution to the time-dependent Schr{\"o}dinger equation
$(-\Delta_x +i\partial_t)u =0$ and this case is treated in
Sj{\"o}gren and Sj{\"o}lin \cite{artikel}. Some additional
examples of appropriate functions $\varphi$ can be found in
\cite{johansson1}.

\par

Let $p\in [1,\infty ]$ and $\omega \in \mathscr P(\rr
{d})$. The (weighted) Fourier Lebesgue space $\mathscr
FL^p_{(\omega )}(\rr d)$ is the inverse Fourier image of
$L^p _{(\omega )} (\rr d)$, i.{\,}e. $\mathscr FL^p_{(\omega )}(\rr d)$
consists of all $f\in \mathscr S'(\rr d)$ such that
\begin{equation}\label{FLnorm}
\nm f{\mathscr FL^{p}_{(\omega )}} \equiv \nm {\widehat f\cdot
\omega }{L^p},
\end{equation}
is finite. If $\omega =1$, then the notation $\mathscr FL^p$
is used instead of $\mathscr FL^p_{(\omega )}$. We note that if
$\omega (\xi )=\eabs \xi ^s$, then $\mathscr
FL^{p}_{(\omega )}$ is the Fourier image of the Bessel potential space
$H^p_s$ (cf. \cite{BL}).

\par

Here and in what follows we use the notation $\mathscr F L^p_s =\mathscr F L^p_{(\omega )}$  when $\omega (\xi )=\eabs \xi ^s$ so
\begin{equation}\label{Fourier-lebesgue}
\|f\|^p_{\mathscr F L^p_s (\rr d)}\equiv \int_{\rr
d}\eabs{\xi}^{sp} |\widehat{f} (\xi )|^p  \, d \xi < \infty.
\end{equation}
\begin{thm} \label{a>1} Assume that the function
$\gamma:\R_+\rightarrow\R_+$ is strictly
increasing and continuous such that $\gamma(0)=0$. Let $R>0$, and let $\varphi$
be real-valued functions on $\rr d$ such that $\varphi '(r,\omega)$ and $\varphi''(r,\omega)$ are continuous and satisfy \eqref{funktionerna} and \eqref{deriverade funktioner} when $r>R$. Also let $p\in (1,\infty)$. Then there
exists a function $f \in \mathscr F L^p_s (\rr d)$, where $s=d(p-1)/p$, such that $S^{\varphi} f$
is continuous in $\{(x,t); t>0\}$ and
\begin{equation}\label{Sa till infty}
\limsup_{(y,t)\rightarrow (x,0)} |S^{\varphi} f(y,t)|= +\infty
\end{equation}
for all $x \in \rr d$, where the limit superior is taken over
those $(y,t)$ for which $|y-x|<\gamma (t)$ and $t>0$.

\par

Furthermore, if $p=1$ then the corresponding result holds for any $s<0$ and if $p=\infty$ the result holds for $s=d$.
\end{thm}
Here we recall that $\varphi'=\varphi'_r$ and $\varphi'=\varphi''_{rr}$ are the first and second orders radial derivatives of $\varphi$.
When $p\in (1,\infty)$ and $s>d(p-1)/p$, $p=1$ and $s\geq 0$, or $p=\infty$ and $s> d$, no counter example of the form in Theorem \ref{a>1} can be provided, since $S^{\varphi}f(y, t)$ converges to $f(x)$ as $(y,t)$ approaches $(x,0)$ non-tangentially when $f \in \mathscr F L^p_s (\rr d)$. In fact, H{\"o}lder's inequality gives
\begin{gather*}
(2 \pi)^d|S^{\varphi}f(x,t)|\le \int_{\rr d}|\widehat{f}(\xi)| \, d\xi\le \nm{\eabs \cdot^{-s}}{L^{p'}(\rr d)}\nm{\widehat f \eabs \cdot^s}{L^p(\rr d)}
\\[1 ex]
\leq \nm{\eabs \cdot^{-s}}{L^{p'}(\rr d)}\nm{f}{\mathscr{F}\! L^p_s(\rr d)},
\end{gather*}
which is finite when $f\in \mathscr F L_s^p(\rr d), \, s>d(p-1)/p$
for some $p\in(1,\infty)$. Here $p'$ is the conjugate exponent, i.e. $1/p+1/p'=1$. Therefore convergence along vertical
lines can be extended to convergence regions when $s>d(p-1)/p$ and
$f$ belongs to $\mathscr F L_s^p(\rr d)$. We note that the estimates still hold for $p=1$ and $p=\infty$, however it follows directly from the first inequality that $S^{\varphi}f(x,t)$ is finite for $f\in \mathscr F L_s^1(\rr d)$ and $s\geq 0$. For $p=\infty$ we have that $S^{\varphi}f(x,t)$ is finite for $f\in \mathscr F L_s^\infty(\rr d)$ and $s> d$.

\par

\section{Notation for the proofs}
\par

In order to prove Theorem \ref{a>1} we introduce some notations.
Let $B_r(x)$ be the open ball in $\rr d$ with center at $x$ and radius
$r$. Numbers denoted by $C, \,c$ or $C'$ may be different at each
occurrence. We let
\begin{equation*}
\delta _k =\delta_{k,d} \equiv \gamma (1/(k+1)) / \sqrt{d}, \qquad
k \in \N,
\end{equation*}
where $\gamma$ is the same as in Theorem \ref{a>1}. Since
$\gamma$ is strictly increasing it is clear that $(\delta
_k)_{k\in \N}$ is strictly decreasing.
We also let $(x_j)_{j=1}^{\infty}\subset \rr d$ be chosen such that
$x_1, x_2,\dots , x_{m_1}$ denotes all points in $B_{1}(0)\cap
\delta_{1} \zz d$, $x_{m_1+1},\dots , x_{m_2}$ denotes all points in $B_{2}(0)\cap
\delta_{2} \zz d$ and generally
\begin{equation*}
\{x_{m_k+1}, \dots , x_{m_{k+1}}\}=B_{k+1}(0)\cap
\delta_{k+1} \zz d,\qquad \text{ for } k\ge 1.
\end{equation*}
Furthermore we choose a strictly decreasing sequence $(t_j)_1^{\infty}$ such that
$
1>t_1>t_2>\cdots > 0
$
and
\begin{equation*}
\frac{1}{k+2}<t_j< \frac{1}{k+1},\qquad k\in \N,
\end{equation*}
for $m_k +1\le j \le m_{k+1}$.

In the proof of
Theorem \ref{a>1} we consider the function $f_{\varphi}$, which
is defined by the formula
\begin{equation}\label{f^}
\widehat{f}_{\varphi}(\xi )=\widehat{f}_{\varphi,B}(\xi )= | \xi | ^{-d} (\log | \xi | )^{-B}
\sum_{j=1}^{\infty} \chi _j(\xi)e^{- i( x_{j} \cdot \xi + t_{j}
{\varphi}(\xi))},
\end{equation}
where for $p\in (1,\infty]$ fixed, we may fixate $B$ such that
$1/p<B<1$. For $p=1$ any $B$, such that $0<B<1$, suffices. We also have that $\chi_j$ is the characteristic
function of
\begin{equation*}
 \Omega_j =\{ \xi \in \rr d;
R_j<|\xi|<R'_j\}.
\end{equation*}

\par
\vspace{0.1 cm}

Here $(R_j)_1^{\infty}$ and
$(R'_j)_1^{\infty}$ are sequences in $\R$ which fulfill the
following conditions:
\begin{enumerate}
\item $R_1 \ge 2+R$, $R'_1\ge R_1+1$, with $R$ given by Theorem \ref{a>1};
\\[1 ex]

\item $R'_j = R_j^N$ when $j\ge 2$, where $N$ is a large positive number and independent of $j$, which is specified
later on;\\[1 ex]
\item $R_j<R'_j <R_{j+1}$, when $j\ge 1$;
\\[1 ex]
\item \begin{equation}\label{vaxande}
|\varphi'(r,\omega)|>1 \qquad \text{when} \qquad r\ge R;
\end{equation}
\\[1 ex]
\item for $j\ge 2$
\begin{equation}\label{Rj1<a<2} R_j^{\min(\beta,1)} >\max_{l<j} \frac{
2^j}{t_l-t_j},
\end{equation}
where $\beta >0$ is the same constant as in \eqref{deriverade funktioner} and
\begin{equation}\label{Rj1<a<2.3} \inf_{R_j\le r\le R_j'}(\inf_{|\omega|=1}|\varphi'(r, \omega)|) >\max_{l<j} \frac{2|x_l-x_j|}{t_l-t_j}.
\end{equation}\\[1 ex]
\end{enumerate}
\begin{rem} The sequences $(R_j)_1^{\infty}$ and $(R_j')_1^{\infty}$ can be chosen since $\varphi$ satisfies condition \eqref{funktionerna}.
\end{rem}
Furthermore, in order to get convenient approximations of the
operator $S^{\varphi}$, we let
\begin{equation}\label{S^a_mf}
S^{\varphi}_mf(x,t) = \frac{1}{(2 \pi)^d} \int_{|\xi |< R'_m}
e^{i x\cdot \xi}e^{it \varphi(\xi)}\widehat{f}(\xi)\,  d \xi.
\end{equation}
Then
\begin{gather}\label{S^a_mf2}
S^{\varphi}_m f_{\varphi}(x,t) =  \sum_{j=1}^{m}A^{\varphi}_j(x, t),
\end{gather}
where
\begin{gather}\label{A^a_j}
A^{\varphi}_j(x, t) =\frac{1}{(2 \pi)^d} \int_{\Omega_j} e^{i
(x-x_j)\cdot \xi}e^{i(t-t_j) \varphi(\xi)} | \xi | ^{-d} (\log |
\xi | )^{-B}\,  d \xi.
\end{gather}
By using polar coordinates we get
\begin{gather}\label{variabelbyte}
A^{\varphi}_j(x_k,t_k) =\frac{1}{(2 \pi)^d} \int_{|\omega| =1}
\Big\{\int_{R_j}^{R'_j} \frac{1}{r(\log r)^{B}}
e^{iF_{\varphi}(r,\omega)}\,  d r\Big\}\,  d \sigma (\omega  ),
\end{gather}
where
\begin{equation*}
F_{\varphi}(r,\omega)= r(x_k - x_j) \cdot \omega + (t_k
-t_j)\varphi(r,\omega),
\end{equation*}
and $d \sigma (\omega)$ is the euclidean surface measure on the $d-1$-dimensional unit sphere.
By differentiation we get
\begin{equation}\label{Fa',a>2}
F'_{\varphi}(r,\omega)= (x_k - x_j) \cdot \omega + (t_k
-t_j)\varphi'(r,\omega)
\end{equation}
and
\begin{equation}\label{Fbiss}
F''_{\varphi}(r,\omega)= (t_k -t_j)\varphi''(r,\omega).
\end{equation}
Here recall that $F'_{\varphi}(r\omega)=F'_{\varphi}(r,\omega)$ and $F''_{\varphi}(r,\omega)$
denote the first and second orders of derivatives of
$F_{\varphi}(r,\omega)$ with respect to the $r$-variable.

\par
\vspace{0.1 cm}

By integration by parts in the inner integral of \eqref{variabelbyte} we get
\begin{gather}\label{a>=2, integral for n>1}
\int_{R_j}^{R'_j} \frac{1}{r (\log r )^{B}} e^{i F_{\varphi}(r,\omega)}\,
d r =
A_{\varphi}-B_{\varphi},
\end{gather}
where
\begin{gather}\label{Del 1 av partialint}
A_{\varphi}=
\Big[\frac{e^{i F_{\varphi}(r,\omega)}}{r (\log r )^{B}i F_{\varphi}'(r,\omega)}\Big
]^{R'_j}_{R_j}
\end{gather}
and
\begin{gather}\label{Del 2 av partialint}
B_{\varphi}=\int_{R_j}^{R'_j} \frac{d}{dr}\Big (\frac{1}{r
(\log r )^{B}i F_{\varphi}'(r,\omega)}\Big )e^{i F_{\varphi}(r,\omega)}\,  d r.
\end{gather}

\par

\section{Proofs}

\par

To prove Theorem \ref{a>1} we need some preparing lemmas. In the following lemma we prove that for fixed $ x\in
B_k(0)$ there exists sequences $(x_{n_j})_{1}^{\infty}$ and
$(t_{n_j})_{1}^{\infty}$ such that
\begin{equation*}
x_{n_j} \in \{x_{m_k+1}, \dots ,x_{m_{k+1}}\},\qquad \text{ and } \qquad t_{n_j} \in \{t_{m_k+1}, \dots ,t_{m_{k+1}}\}
\end{equation*}
 and $|x_{n_j}-x|<\gamma(t_{n_j})$. The lemma is left without
 proof since the result can be found in \cite{johansson1}.
\begin{lem}\label{dense}
Let $x\in \rr d$ be fixed. Then for each $k \ge |x|$ there exists
$x_{n_j} \in \{x_{m_k+1}, \dots ,x_{m_{k+1}}\}$ and $t_{n_j} \in \{t_{m_k+1}, \dots ,t_{m_{k+1}}\}$ such that
$|x_{n_j}-x|<\gamma(t_{n_j})$. In particular $(x_{n_j},t_{n_j})\to
(x,0)$ as $j$ turns to infinity.
\end{lem}
\par

We want to prove that $f_{\varphi}$ in \eqref{f^} belongs to
$\mathscr F L_s^p(\rr d)$, with $s=d(p-1)/p$, and fulfill \eqref{Sa till infty}. The former
relation is a consequence of Lemma \ref{ghat} below, which
concerns Sobolev space properties for functions of the form
\begin{equation}\label{definitiong}
\widehat{g}(\xi )= | \xi | ^{-d} (\log | \xi | )^{-\rho/p}
\sum_{j=1}^{\infty} \chi _j(\xi)b_j(\xi),
\end{equation}
where $\chi_j$ is the characteristic function on disjoint sets $\Omega_j$.

\begin{lem} \label{ghat} Assume that $p\in(1,\infty)$, $\rho >1$, $\Omega_j$ for $j \in \N$ are disjoint open subsets of
$\rr d\backslash B_{\mu}(0)$ for some $\mu>2$, $b_j\in L^1_{loc}(\rr d)$ for $j \in
\N$ satisfies
\begin{equation*}
\sup_{j\in \N}\|b_j\|_{L^{\infty}(\Omega_j)}<\infty ,
\end{equation*}
and let $\chi_j$ be the characteristic function for $\Omega_j$. If
$g$ is given by \eqref{definitiong}, then $g\in \mathscr F
L_s^p(\rr d)$, where $s=d(p-1)/p$.

\par

For $p=1$ and $\rho \geq 0$ it follows that $g\in \mathscr F \! L_s^1(\rr d)$ for any $s<0$. Furthermore, if $p=\infty$ and $\rho/p$ in \eqref{definitiong} is replaced by $\rho\geq 0$, then $g\in \mathscr F \! L_s^\infty(\rr d)$ for $s=d$.
\end{lem}
\begin{proof} By estimating \eqref{Fourier-lebesgue} for the function $g$ when $p\in(1,\infty)$, $s=d(p-1)/p$ and the assumptions given by the lemma is satisfied, we get that
\begin{multline*}
\int_{\rr d}|\widehat{g}(\xi )|^p\eabs{\xi}^{sp}\,  d\xi \cr \le C
\int_{\rr d \setminus B_{\mu}(0)}| \xi | ^{-dp} (\log | \xi |
)^{-\rho} \eabs{\xi}^{sp}\,  d\xi  \cr \le 2^{sp/2}C
\int_{\mu}^{\infty}\frac{1}{r(\log r  )^{\rho}} \,  dr < \infty .
\end{multline*}
The second inequality holds since
$(1+r^2)^{sp/2}<(r^2+r^2)^{sp/2}= 2^{sp/2} r^{sp}$ for $r>1$.

\par

For $p=1$ and $s<0$, the second inequality, in the estimates above, is replaced by 
\begin{equation*}
\int_{\rr d \setminus B_{\mu}(0)}| \xi | ^{-d} (\log | \xi |)^{-\rho} \eabs{\xi}^{s}\,  d\xi\\[1 ex] \le 2^{s/2}C
\int_{\mu}^{\infty}\frac{1}{r^{1-s}(\log r  )^{\rho}} \,  dr< \infty
\end{equation*}

\par

For $p=\infty$ we have the norm 
\begin{multline*}
\operatorname{ess \ sup }(|\widehat{g}(\xi )|^p\eabs{\xi}^{sp})\le C
\operatorname{ess \ sup }_{|\xi|>\mu}(| \xi | ^{-d} (\log | \xi |
)^{-\rho} \eabs{\xi}^{d}) \\[1 ex] \le 2^{d/2} (\log |\mu |
)^{-\rho} < \infty .
\end{multline*}

\par

\end{proof}
In the following lemma we give estimates of the expression $A_j^{\varphi}$.
\begin{lem}\label{uppskattning} Let $A_j^{\varphi}(x,t)$ be given by \eqref{A^a_j}.
Then the following is true:
\begin{alignat*}{2}
(1)& \qquad \sum_{j=1}^{k-1} |A^{\varphi}_j(x, t)| \le C (\log
R'_{k-1})^{1-B}, \textit{ with } C \textit{ independent of } k;
\\[1 ex] (2)& \qquad A^{\varphi}_k(x_k,t_k)>c(\log
R'_k)^{1-B}, \textit{ with } c>0 \textit{ independent of } k.
\end{alignat*}
\end{lem}
\begin{proof}
$(1)$ By triangle inequality and the fact that $|\xi |> 2$, when
$\xi\in \Omega_j$, we get
\begin{multline*}
 \sum_{j=1}^{k-1} |A^{\varphi}_j(x, t)| \le
\frac{1}{(2 \pi)^d}\int_{2 \le |\xi | \le R'_{k-1}} | \xi |^{-d}
(\log | \xi |)^{-B}\,  d \xi\cr = C
\int_2^{R'_{k-1}}\frac{1}{r(\log r)^{B}}\,  d r \le C (\log
R'_{k-1})^{1-B},
\end{multline*}
where $C$ is independent of $k$. In the last equality we have
taken polar coordinates as new variables of integration.

\par

 \flushleft{$(2)$} Since $R^N_j=R'_j$ for sufficiently large $N$,
we get
\begin{multline*}
A^{\varphi}_k(x_k,t_k)= C \int_{R_k}^{R'_k} \frac{1}{r(\log
r)^{B}}\,  d r \cr = C \Big ((\log R'_k)^{1-B}-(\log
(R'_k)^{1/N})^{1-B}\Big ) \cr = C\Big(1-\frac{1}{N^{1-B}}\Big)
(\log R'_k)^{1-B}
>c(\log R'_k)^{1-B},
\end{multline*}
for some constant $c>0$, which is independent of $k$.
\end{proof}

\begin{lem}\label{kontinuitet} Assume that $S^{\varphi}_m f_{\varphi}$ is given by \eqref{S^a_mf}. Then $S^{\varphi}_m f_{\varphi}$ is continuous on
$\{(x,t);t>0, x\in \rr d\}$.
\end{lem}
\begin{proof}
 The continuity for each $S^{\varphi}_m f_{\varphi}$ follows from the facts, that for almost every $\xi\in \rr d$, the map
 \begin{equation*}
 (x,t)\mapsto e^{i x\cdot \xi}e^{i t {\varphi}(\xi)}
\widehat{f}_{\varphi} ( \xi )
\end{equation*}
is continuous, and that
 \begin{equation*}\int_{|\xi|<R'_m}
|e^{i x\cdot \xi}e^{i t {\varphi}(\xi)} \widehat{f}_{\varphi} ( \xi )| \, d\xi =
\int_{|\xi|<R'_m} |\widehat{f}_{\varphi} ( \xi )| \, d\xi <C.
\end{equation*}
\end{proof}

When proving Theorem \ref{a>1}, we first prove that the modulus
of $S^{\varphi}_mf_{\varphi}(x_k, t_k)$ turns to infinity as $k$
goes to infinity. For this reason we note that the
triangle inequality and \eqref{S^a_mf2} implies that
\begin{multline}\label{summa}
|S^{\varphi}_mf_{\varphi}(x_k, t_k)|\ge |A^{\varphi}_k(x_k,t_k)|-\Big|\sum_{j
=1}^{k-1}A^{\varphi}_j(x_k,t_k)\Big| -\Big|\sum_{j=k+1}^m
A^{\varphi}_j(x_k,t_k) \Big|,
\end{multline}
where $m>k$.
We want to estimate the terms in \eqref{summa}. From Lemma
\ref{uppskattning} we get estimates for the first two terms. It
remains to estimate the last term.

\par

\begin{proof}[Proof of Theorem \ref{a>1}.]$\ $

\vspace{-0.15 cm}

{\flushleft{\textbf{Step $\mathbf{1}$.}}}  For $j>k \ge
2$ we shall estimate $|A_j^{\varphi} (x_k,t_k)|$ in \eqref{variabelbyte}.
We have to find appropriate estimates for $A_{\varphi}$ and
$B_{\varphi}$ in \eqrefs{a>=2, integral for n>1}{Del 1 av partialint}{Del 2 av partialint}. By using $t_k - t_j
> 0$ and $ \ R_j< r <R'_j$ it follows from \eqref{Rj1<a<2.3}, \eqref{Fa',a>2}, triangle inequality and Cauchy-Schwarz inequality that
\begin{multline}\label{absfprima}
|F'_{\varphi}(r,\omega)| \ge  (t_k -t_j)
|\varphi'(r,\omega)|-|x_k - x_j|\cr
>(t_k -t_j)|\varphi'(r,\omega)|-(t_k -t_j)\frac{|\varphi'(r,\omega)|}{2}\cr= \frac{|\varphi'(r,\omega)|}{2}(t_k -t_j).
\end{multline}
From \eqref{vaxande}, \eqref{Rj1<a<2} and
\eqref{absfprima} it follows that
\begin{multline*}
|A_{\varphi}| = \Big |  \Big [\frac{1}{r (\log r )^{B}i
F'_{\varphi}(r,\omega)}e^{i F_{\varphi}(r,\omega)}\Big
]^{R'_j}_{R_j}\Big | \cr  \le \frac{C}{R_j}\Big(\frac{1}{|F'_{\varphi}(R_j,\omega)|}+ \frac{1}{|F'_{\varphi}(R'_j,\omega)
|}\Big)\le \frac{C}{(t_k-t_j)R_j} \le C 2^{-j}.
\end{multline*}

\par

In order to estimate $B_{\varphi}$, using \eqref{deriverade funktioner}, \eqref{Fbiss} and \eqref{absfprima}, we have
\begin{multline*}
 \Big | \frac{d}{dr}\Big (\frac{1}{r (\log r )^{B}i
F'_{\varphi}(r,\omega)}\Big )e^{i F_{\varphi}(r,\omega)}\Big |\cr
\le \frac{C}{r^2|F'_{\varphi}(r,\omega)|}  +
\frac{C|F''_{\varphi}(r,\omega)|}{r|F'_{\varphi}(r,\omega)|^2(\log
r )^{B}} < \frac{C}{r^{1+\min(1,\beta)}(t_k-t_j)}.
\end{multline*}
This together with \eqref{Rj1<a<2} gives us
\begin{multline*}
|B_{\varphi}| = \Big |\int_{R_j}^{R'_j} \frac{d}{dr}\Big
(\frac{1}{r (\log r )^{B}i F'_{\varphi}(r,\omega)}\Big )e^{i
F_{\varphi}(r,\omega)}\,  d r \Big | \cr \le
\int_{R_j}^{R'_j}\frac{C}{r^{1+\min(1,\beta)}(t_k-t_j)}\,  d r \le
\frac{C}{R_j^{\min(1,\beta)}(t_k-t_j)} \le C 2^{-j}.
\end{multline*}
From the estimates above and the triangle inequality we get
\begin{gather}\label{Ajxk}
|A^{\varphi}_j(x_k, t_k)|\le
C(|A_{\varphi}|+|B_{\varphi}|)< C2^{-j}, \qquad j> k \ge 2.
\end{gather}
 Here $C$ is independent of $j$ and $k$.

\par

 Using the results from
\eqref{summa}, \eqref{Ajxk}, in combination with Lemma
\ref{uppskattning}, and recalling that $R_j'=R_j^N$, gives us
\begin{multline}\label{slutresultat}
|S^{\varphi}_m f_{\varphi}(x_k, t_k)| \ge c (\log
R'_k)^{1-B}- C'(\log R_k)^{1-B}- C\sum _{k+1}^m 2^{-j}\cr \ge c(\log(R_k'))^{1-B}-\frac{C'}{N^{1-B}}(\log(R_k'))^{1-B}-C \ge c
(\log R'_k)^{1-B},
\end{multline}
 when $m>k$ and $N$ is chosen sufficiently large. Here $c>0$ is independent of $k$.

\par
{\flushleft{\textbf{Step $\mathbf{2}$.} }} Now it remains to show
that $S^{\varphi} f_{\varphi}$ is continuous when $t>0$, and then
it suffices to prove this continuity on a compact subset $L$ of
$$\{(x,t);\, t>0,\, x\in \rr d\}.$$ We want to replace $(x_l, t_l)$
with $(x,t)\in L $ in \eqref{Rj1<a<2} and \eqref{Rj1<a<2.3}. Since we have maximum over all $l$ less than $j$, we can choose
$j_0<\infty$ large enough such that for all $j>l>j_0$ we have that $t_j<t_l<t$. Hence we may replace $(x_l, t_l)$
with $(x,t)\in L $ on the right-hand sides in \eqref{Rj1<a<2} and \eqref{Rj1<a<2.3} for all $j>j_0$.
This in turn implies that \eqref{Ajxk} holds when $(x_k, t_k)$ is replaced by $(x,
t) \in L$ and $j>j_0$. We use \eqref{Ajxk} to conclude that
\begin{multline*}
|S^{\varphi}_m f_{\varphi}(x,t) -S^{\varphi} f_{\varphi}(x,t)|\cr = \Big
|(2 \pi)^{-d} \int_{|\xi |< R'_m} e^{i x\cdot \xi}e^{it
\varphi(\xi)}\widehat{f}_{\varphi}(\xi)\,  d \xi - (2
\pi)^{-d} \int_{\rr d} {e^{i x\cdot \xi}e^{it
\varphi(\xi)}\widehat{f}_{\varphi}(\xi)}\,  d \xi \Big |\cr
=(2 \pi)^{-d} \Big | \int_{|\xi |> R'_m} e^{i x\cdot
\xi}e^{it \varphi(\xi)}\widehat{f}_{\varphi}(\xi)\,  d \xi \Big |
\le C\sum_{i=m+1}^{\infty}  2^{-i} = C2^{-m},
\end{multline*}
when $m> j_0$.
Hence $S^{\varphi}_m f_{\varphi}$ converge uniformly to $S^{\varphi} f_{\varphi}$ on every compact set.

\par
We have now showed that
$S^{\varphi}_m f_{\varphi}$ converge uniformly to $S^{\varphi}
f_{\varphi}$ on every compact set and from Lemma \ref{kontinuitet}
it follows that each $S^{\varphi}_m f_{\varphi}$ is a continuous
function. Therefore it follows that $S^{\varphi} f_{\varphi}$ is
continuous on $\{(x,t);t>0\}$. In particular there is an $N \in
\N$ such that
\begin{equation*}|S^{\varphi}_m
f_{\varphi}(x_k, t_k)-S^{\varphi}
f_{\varphi}(x_k,t_k)| < 1,
\end{equation*}
when $m>N$.
Using \eqref{slutresultat}
and the triangle inequality we get
\begin{multline*} c(\log R'_k)^{1-B} \le
|S^{\varphi}_m f_{\varphi}(x_k, t_k) | \cr \le
|S^{\varphi}_m f_{\varphi}(x_k, t_k)-S^{\varphi}
f_{\varphi}(x_k,t_k)| + |S^{\varphi} f_{\varphi}
(x_k,t_k)| <\cr 1+ |S^{\varphi}
f_{\varphi}(x_k,t_k)|.
\end{multline*}
This gives us
\begin{equation*}
|S^{\varphi} f_{\varphi}(x_k,t_k)| > c(\log R'_k)^{1-B}-1
\rightarrow +\infty \text{ \ as \ } k \rightarrow +\infty.
\end{equation*}
For any fixed $x\in \rr d$ we can by Lemma \ref{dense} choose a subsequence $(x_{n_j},t_{n_j})$ of $(x_k,t_k)$ that goes to $(x,0)$ as $j$ turns to infinity. This gives the result.
\end{proof}

\section{Mixed Fourier Lebesgue spaces}
In this section we consider weighted mixed Fourier Lebesgue spaces
as initial datas for the generalized free time-dependent
Schrödinger equation. An analogous version of the previous results
holds. Due to similarities in the proofs, as well as for the
definition of initial data, we only need to show that the initial
data belongs to the mixed Fourier Lebesgue space.

\par

Let $p, q\in [1,\infty ]$ and $d_1+d_2=d$, where $d_1,d_2\in
\mathbf N$. From now on we consider weights
$\omega(\xi)=\omega(\xi_1,\xi_2)$, for $\xi_1\in \rr{d_1}$ and
$\xi_2\in \rr{d_2}$, which are positive and $L^1_{\loc}(\rr d)$.
The (weighted) mixed Fourier Lebesgue space $\mathscr
F\!L^{p,q}_{(\omega )}(\rr d)=\mathscr F\!L^{p,q}_{(\omega )}(\rr
{d_1}\times \rr {d_2})$ consists of all $f\in \mathscr S'(\rr d)$
such that
\begin{multline}\label{FLnorm1}
\nm f{\mathscr FL^{p,q}_{(\omega )}(\rr d)} \equiv \nm {\widehat
f\cdot \omega }{L^{p,q}}\\[1 ex]\equiv
\Big(\int_{\rr{d_2}}\Big(\int_{\rr{d_1}} |\widehat{f}
(\xi_1,\xi_2)|^p\omega(\xi_1,\xi_2)^p \, d\xi_1\Big)^{q/p}\,
d\xi_2\Big)^{1/q},
\end{multline}
is finite. If $\omega =1$, then the notation $\mathscr FL^{p,q}$
is used instead of $\mathscr FL^{p,q}_{(\omega )}$.

\par

Here and in what follows we use the notation $\mathscr F L^{p,q}_{(\omega )}=\mathscr F L^{p,q}_{s_1,s_2}$  when $\omega (\xi_1,\xi_2 )=\eabs {\xi_1} ^{s_1}\eabs {\xi_2} ^{s_2}$ so
\begin{multline}\label{Fourier-lebesgue1}
\|f\|^q_{\mathscr F L^{p,q}_{s_1,s_2} (\rr d)}\\[1 ex]\equiv \int_{\rr {d_2}}\eabs{\xi_2}^{s_2q}\Big(\int_{\rr {d_1}}\eabs{\xi_2}^{s_1p}
|\widehat{f} (\xi_1,\xi_2)|^p  \, d \xi_1\Big)^{q/p}  \, d \xi_2< \infty.
\end{multline}
\begin{thm} \label{a>1,1} Assume that the function
$\gamma:\R_+\rightarrow\R_+$ is strictly
increasing and continuous such that $\gamma(0)=0$. Let $R>0$, and let $\varphi$
be real-valued functions on $\rr d$ such that $\varphi '(r,\omega)$ and $\varphi''(r,\omega)$ are continuous and satisfy \eqref{funktionerna} and \eqref{deriverade funktioner} when $r>R$. Also let $p,q\in (1,\infty)$ with $1/p +1/q<1$. Then there
exists a function $f \in \mathscr F L^{p,q}_{s_1,s_2} (\rr d)$, where $s_1=d_1(p-1)/p$ and $s_2=d_2(q-1)/q$, such that $S^{\varphi} f$
is continuous in $\{(x,t); t>0\}$ and
\begin{equation}\label{Sa till infty1}
\limsup_{(y,t)\rightarrow (x,0)} |S^{\varphi} f(y,t)|= +\infty
\end{equation}
for all $x \in \rr d$, where the limit superior is taken over those $(y,t)$ for which $|y-x|<\gamma (t)$ and $t>0$.
Note that for $p=\infty$ or $q=\infty$ the result holds when $s_1=d_1$ or $s_2=d_2$, respectively.
\end{thm}

\begin{rem}

Note that $\mathscr F \! L^{p,p}=\mathscr F \! L^p$ holds
independently of $d_1$ and $d_2$. However the weights used in this
section for mixed Lebesgue spaces,
$\eabs{\xi_1}^{s_1}\eabs{\xi_2}^{s_2}$ is not equivalent to those
used for the $L^p$ spaces, even when $p=q$. In fact, by choosing
$p=q$ for these mixed spaces the weights are larger then the
weights used for the usual Lebesgue spaces, and thereby the
results in this section for the special case $p=q$ concern finding
counter examples in a smaller space. It is here important to note
that the result of previous section are not completely contained
in the results obtained here, since the condition $1/p+1/q<1$
implies the requirement that $p>2$ if $p=q$. No results for
$p=q<2$ are obtained in this section.

\end{rem}

When $s_1>d_1(p-1)/p$ and $s_2>d_2(q-1)/q$ no counter example of the form in Theorem \ref{a>1,1} can be provided, since $S^{\varphi}f(y, t)$ converges to $f(x)$ as $(y,t)$ approaches $(x,0)$ non-tangentially when $f \in \mathscr F L^{p,q}_{s_1,s_2} (\rr d)$. In fact, H{\"o}lder's inequality gives
\begin{gather*}
(2 \pi)^d|S^{\varphi}f(x,t)|\le \int_{\rr d}|\widehat{f}(\xi)| \, d\xi =\int_{\rr {d_2}}\int_{\rr {d_1}}|\widehat{f}(\xi_1,\xi_2)| \, d\xi_1\, d\xi_2\\[1 ex] \le \nm{\eabs {\cdot}^{-s_2}}{L^{q'}(\rr {d_2})}\nm{\int_{\rr {d_1}}|\widehat{f}(\xi_1,\cdot_2)| \, d\xi_1 \eabs{\cdot_2}^{s_2}}{L^q(\rr {d_2})}
\\[1 ex] \le \nm{\eabs {\cdot}^{-s_2}}{L^{q'}(\rr {d_2})}\nm{ \eabs{\cdot}^{-s_1}}{L^{p'}(\rr {d_1})}\nm{\nm{\widehat{f}(\cdot_1,\cdot_2) \eabs{\cdot_1}^{s_1}\eabs{\cdot_2}^{s_2}}{L^p(\rr {d_1})}}{L^q(\rr {d_2})}
\\[1 ex] = \nm{\eabs {\cdot}^{-s_2}}{L^{q'}(\rr {d_2})}\nm{ \eabs{\cdot}^{-s_1}}{L^{p'}(\rr {d_1})}\nm{\widehat{f}}{L^{p,q}_{s_1,s_2}(\rr d)}
\end{gather*}
which is finite when $f\in \mathscr F L_{s_1,s_2}^{p,q}(\rr d), \,
s_1>d_1(p-1)/p, \, s_2>d_2(q-1)/q$ for some $p,q\in (1,\infty)$. Here $p'$ and $q'$ denotes the conjugate exponents of $p$ and $q$ respectively, i.e. $1/p+1/p'=1$ and $1/q+ 1/q'=1$.
As noticed for $p=1$ in the previous section, different estimates
are used. For $p=q=1$ convergence follows directly from the first
inequality for $f\in \mathscr F\!L^{p}_{s_1,s_2}$, where $s_1\geq
0$ and $s_2\geq 0$. In case $p=1$ and $q\in (0,\infty)$ we have
that
\begin{multline*}
(2 \pi)^d|S^{\varphi}f(x,t)|\le \int_{\rr d}|\widehat{f}(\xi)| \,
d\xi =\int_{\rr {d_2}}\Big(\int_{\rr
{d_1}}|\widehat{f}(\xi_1,\xi_2)| \, d\xi_1 \Big)\, d\xi_2\\[1 ex]
\le \nm{\eabs {\cdot}^{-s_2}}{L^{q'}(\rr {d_2})}\nm{\int_{\rr
{d_1}}|\widehat{f}(\xi_1,\cdot_2)| \, d\xi_1
\eabs{\cdot_2}^{s_2}}{L^q(\rr {d_2})}
\\[1 ex] = \nm{\eabs {\cdot}^{-s_2}}{L^{q'}(\rr {d_2})}\nm{\widehat{f}}{L^{1,q}_{0,s_2}(\rr d)},
\end{multline*}
which is finite when $f\in \mathscr F\!L^{p,q}_{s_1,s_2}(\rr d)$,
$s_1\geq 0$ and $s_2>d_2(q-1)/q$. In a similar way it follows that
for $q=1$ and $p\in (0,\infty)$, for initial data $f\in \mathscr
F\!L^{p}_{s_1,s_2}$ no counter example concerning non-tangential
convergence exist when $s_1>d_1(p-1)/p$ and $s_2\geq 0$. Therefore
convergence along vertical lines can be extended to convergence
regions when $s_1>d_1(p-1)/p, \, s_2>d_2(q-1)/q$ and $f$ belongs
to $\mathscr F L_{s_1,s_2}^{p,q}(\rr d)$, or for $p=1$ and $q=1$,
$s_1\geq 0$ and $s_2\geq 0$, respectively.

\par

In the proof of
Theorem \ref{a>1,1} we consider the function $f_{\varphi}$, which
is defined by the formula
\begin{equation}\label{f^1}
\widehat{f}_{\varphi}(\xi )=\widehat{f}_{\varphi,B}(\xi )= | \xi | ^{-d} (\log | \xi | )^{-B}
\sum_{j=1}^{\infty} \chi _j(\xi)e^{- i( x_{j} \cdot \xi + t_{j}
{\varphi}(\xi))},
\end{equation}
where for $p,q\in (1,\infty)$,  $1/p +1/q<1$ fixed, we may fixate
$0<B<1$ and $0<k<1$ such that $1/p<B k<1$ and $1/q<B (1-k)<1$. We
also have that $\chi_j$ is the characteristic function of
\begin{equation*}
 \Omega_j =\{ \xi \in \rr d;
R_j<|\xi|<R'_j\}.
\end{equation*}

\par

We want to prove that $f_{\varphi}$ in \eqref{f^} belongs to
$\mathscr F L_{s_1,s_2}^{p,q}(\rr d)$ and fulfill \eqref{Sa till infty1}. The former
relation is a consequence of Lemma \ref{ghat1} below, which
concerns Sobolev space properties for functions of the form
\begin{equation}\label{definitiong1}
\widehat{g}(\xi )= | \xi | ^{-d} (\log | \xi | )^{-\rho/p}
\sum_{j=1}^{\infty} \chi _j(\xi)b_j(\xi),
\end{equation}
where $\chi_j$ is the characteristic function on disjoint sets $\Omega_j$.

\begin{lem} \label{ghat1} Assume that $0<\rho <1$ such that there exists $0<k<1$ for which $1/p<\rho k <1$ and $1/q < \rho (1-k)<1$,  $\Omega_j$ for $j \in \N$ are disjoint open subsets of
$\rr d\backslash B_{\mu}(0)$, where $\mu>2$, $b_j\in L^1_{loc}(\rr d)$ for $j \in
\N$ satisfies
\begin{equation*}
\sup_{j\in \N}\|b_j\|_{L^{\infty}(\Omega_j)}<\infty ,
\end{equation*}
and let $\chi_j$ be the characteristic function for $\Omega_j$. If $g$ is given by \eqref{definitiong1}, then
$g\in \mathscr F L_{s_1,s_2}^{p,q}(\rr d)$, where $s_1=d_1(p-1)/p$ and $s_2=d_2(q-1)/q$.

For $p=\infty$ or $q=\infty$ the result holds for $s_1=d_1$ or $s_2=d_2$, respectively. 
\end{lem}
\begin{proof} Here we give the proof for $1<p,q<\infty$. The modifications for $p=\infty$ and $q=\infty$ are left for the reader. By estimating \eqref{Fourier-lebesgue1} for the function $g$ we get that
\begin{multline*}
\int_{\rr {d_2}}\left(\int_{\rr {d_1}}|\widehat{g}(\xi_1,\xi_2 )|^p\eabs{\xi_1}^{s_1p}\,  d\xi_1\right)^{q/p}\eabs{\xi_2}^{s_2q}\,  d\xi_2 \\[1 ex] \le \int_{\rr {d_2}}\left(\int_{\rr {d_1}}| \xi | ^{-dp} (\log | \xi | )^{-\rho p}\eabs{\xi_1}^{s_1p}\,  d\xi_1\right)^{q/p}\eabs{\xi_2}^{s_2q} d\xi_2,\end{multline*}
when $|\xi|>\mu$. Here it is sufficient to prove that the latter expression is finite.
Since $|\xi|$ is bounded from below it follows that $|\xi_1|$ and $|\xi_2|$ can not be arbitrarily small at the same time. The integral is divided into the three following parts
\begin{multline*}
 \int_{|\xi_2|>\mu}\left(\int_{|\xi_1|<\mu}| \xi | ^{-dp} (\log | \xi | )^{-\rho p}\eabs{\xi_1}^{s_1p}\,  d\xi_1\right)^{q/p}\eabs{\xi_2}^{s_2q} d\xi_2\\[1 ex]\leq
  \int_{|\xi_2|>\mu}| \xi_2 | ^{-dq} (\log | \xi_2 | )^{-\rho q}\left(\int_{|\xi_1|<\mu}\eabs{\xi_1}^{s_1p}\,  d\xi_1\right)^{q/p}\eabs{\xi_2}^{s_2q} d\xi_2,\end{multline*}
which is finite since $q(d_1-d)< 0$,
\begin{multline*}
 \int_{|\xi_2|<\mu}\left(\int_{|\xi_1|>\mu}| \xi | ^{-dp} (\log | \xi | )^{-\rho p}\eabs{\xi_1}^{s_1p}\,  d\xi_1\right)^{q/p}\eabs{\xi_2}^{s_2q} d\xi_2\\[1 ex]\leq
  \int_{|\xi_2|<\mu}\left(\int_{|\xi_1|>\mu}| \xi_1 | ^{-dp} (\log | \xi_1 | )^{-\rho p}\eabs{\xi_1}^{s_1p}\,  d\xi_1\right)^{q/p}\eabs{\xi_2}^{s_2q} d\xi_2,\end{multline*}
which is finite since $p(d-d_2)< 0$
and
\begin{multline*}  C \left(\int_{|\xi_1|>\mu}| \xi_1 | ^{-d_1p} (\log | \xi_1 | )^{-\rho k p}
\eabs{\xi_1}^{s_1p}\,  d\xi_1  \right)^{q/p}\\[1
ex]\cdot\left(\int_{|\xi_2|>\mu}| \xi_1 | ^{-d_2q} (\log | \xi_2 |
)^{-\rho (1-k)q} \eabs{\xi_2}^{s_2q}\,  d\xi_2 \right)\\[1 ex]
\le 2^{(s_1p+s_2 q)/2} C \left(\int_{\rho}^{\infty}\frac{1}{r(\log
r  )^{\rho k p}}\right)\left(\int_{\rho}^{\infty}\frac{1}{r(\log r
)^{\rho (1-k)q}}\right) \,  dr < \infty .
\end{multline*}
The second inequality holds since $(1+r^2)^{sp/2}<(r^2+r^2)^{sp/2}= 2^{sp/2}
r^{sp}$ for $r>1$. The last inequality follows by choosing $0< k< 1$ such that $\rho k p>1$ and $\rho(1-k)q >1$.
\end{proof}


\begin{thebibliography}{10}

\bibitem{BL} J. Bergh, J. Löfström, \emph{Interpolation Spaces, An Introduction}, Springer, Berlin (1976).

\bibitem{Bourgain} J. Bourgain, \emph{ A remark on Schr{\"o}dinger operators}, Isr. J. Math.
\textbf{77} (1992), no. 1-2, 1-16.


\bibitem{johansson1} K. Johansson, \emph{A counterexample on nontangential convergence for oscillatory integrals},
 Publications de l'Institut Mathematique \textbf{87} (2010), no. 101, 129--137.

\bibitem{Kenig} C.\,E. Kenig, G. Ponce and L. Vega, \emph{ Oscillatory integrals and regularity of dispersive equations}, Indiana Univ. Math. J. \textbf{40} (1991), no. 1, 33-69.

\bibitem{artikel} P. Sj{\"o}gren and P. Sj{\"o}lin, \emph{Convergence properties for the time-dependent
Schr{\"o}dinger equation}, Ann. Acad. Sci. Fenn. Ser. A I, Math. \textbf{14} (1989), no. 1, 13-25.

\bibitem{ref2} P. Sj{\"o}lin, \emph{Regularity of solutions to the Schr{\"o}dinger equation},
Duke Math. J. \textbf{55} (1987), no. 3, 699-715.

\bibitem{LP} P. Sj{\"o}lin, \emph{$L\sp p$ maximal estimates for solutions to the
Schr{\"o}dinger equation}, Math. Scand. \textbf{81} (1997), no. 1, 35-68.

\bibitem{Counter} P. Sj{\"o}lin, \emph{A counter-example concerning maximal estimates for solutions to equations of Schr{\"o}dinger type}, Indiana Univ. Math. J. \textbf{47} (1998), no. 2, 593-599.

\bibitem{Hom} P. Sj{\"o}lin, \emph{Homogeneous maximal estimates for solutions to the Schr{\"o}dinger equation}, Bull. Inst. Math. Acad. Sin. \textbf{30} (2002), no. 2, 133-140.

\bibitem{SteinWeiss} E.\,M. Stein and G. Weiss, \emph{Introduction to Fourier analysis on Euclidean spaces}, Princeton Mathematical Series. Princeton, New Jersey, 1971.

\bibitem{Sharp} B.\,G. Walther, \emph{Sharpness results for $L\sp 2$-smoothing of oscillatory integrals}, Indiana Univ. Math. J. \textbf{50} (2001), no. 1, 655-669.

\bibitem{Sharpmax}  B.\,G. Walther, \emph{Sharp maximal estimates for doubly oscillatory integrals}, Proc. Am. Math. Soc. \textbf{130} (2002), no. 12, 3641-3650.

\end{thebibliography}
\end{document}